\documentclass[12 pt]{article}
\usepackage{amsfonts,amsmath,color,amsthm,amssymb}

\title{Sharp Threshold Asymptotics for the Emergence of Additive Bases}
\author
{Anant Godbole\\
Department of Mathematics and Statistics\\
East Tennessee State University\and Chang Mou Lim\\
Department of Mathematics\\
University of Chicago
\and
Vince Lyzinski\\
Department of Applied Mathematics and Statistics\\
The Johns Hopkins University\and
Nicholas George Triantafillou\\
Department of Mathematics\\
University of Michigan, Ann Arbor}
\begin{document}
\def\ep{\varepsilon}
\def\aln{\alpha n}
\def\maln{(2-\alpha)n}
\def\lr{\left(}
\def\lc{\left\{}
\def\rc{\right\}}
\def\rr{\right)}
\def\nt{{n\choose2}}
\def\nc{\lceil n/2\rceil}
\def\ic{\lceil i/2\rceil}
\def\nr{{n\choose r}}
\def\ns{{n\choose s}}
\def\tv{{d_{{\rm TV}}}}
\def\P{{\rm Po}}
\def\cl{{\cal L}}
\def\ta{{\tilde{a}}}
\def\tb{{\tilde{b}}}
\def\qed{\vbox{\hrule\hbox{\vrule\kern3pt\vbox{\kern6pt}\kern3pt\vrule}\hrule}}
\def\ca{{\cal A}}
\def\vca{{\vert\cal A\vert}}
\def\ck{{\cal K}}
\def\ch{{\cal H}}
\def\cc{{\cal C}}
\def\cs{{\cal S}}
\def\p{\mathbb P}
\def\v{\mathbb V}
\def\e{\mathbb E}
\def\z{\mathbb Z}
\def\l{\lambda}
\def\a{\alpha}
\def\exp{\mathrm{exp}}
\def\n{\noindent}
\newtheorem{thm}{Theorem}
\newtheorem{lm}[thm]{Lemma}
\newtheorem{rem}[thm]{Remark}
\newtheorem{exam}[thm]{Example}
\newtheorem{prop}[thm]{Proposition}
\newtheorem{defn}[thm]{Definition}
\newtheorem{cm}[thm]{Claim}
\maketitle
\begin{abstract}
A set $\ca\subseteq[n]\cup\{0\}$ is said to be a 2-additive basis  for $[n]$  if each $j\in[n]$ can be written as $j=x+y, x,y\in\ca, x\le y$.  If we pick each integer in $[n]\cup\{0\}$ independently with probability $p=p_n\to0$, thus getting a random set $\ca$, what is the probability that we have obtained a 2-additive basis?  We address this question when the target sum-set is $[(1-\alpha)n,(1+\alpha)n]$ (or equivalently $[\alpha n, (2-\alpha) n]$) for some $0<\alpha<1$.  Under either model, the Stein-Chen method of Poisson approximation is used, in conjunction with Janson's inequalities, to tease out a very sharp threshold for the emergence of a 2-additive basis.  Generalizations to $k$-additive bases are then given.  
\end{abstract}
\section{Introduction}  In 1956, Erd\H os \cite{erdos} answered a question posed in 1932 by Sidon by proving that there exists an infinite sequences of natural numbers ${\cal S}$ and constants $c_1$ and $c_2$ such that for large $n$, \begin{equation}c_1\log n\le r_2(n)\le c_2\log n,\end{equation}
where, for $k\ge 2$, $r_k(n)$ is the {\it number of ways} of representing the integer $n$ as the sum of $k$ elements from ${\cal S}$, a so-called {\it asymptotic basis} of order $k$.  The result was generalized in the 1990 work of Erd\H os and Tetali \cite{et} which established that there exists an infinite sequence ${\cal S}$ for which (1) was true for each fixed $k\geq 2$, i.e., for each large $n$, 
\begin{equation} r_k(n)=\Theta(\log n). 
\end{equation}
To achieve this result, Erd\H os and Tetali constructed a random sequence $\mathcal S$ of natural numbers by including $z$ in $\mathcal S$ with probability
$$p(z)=\begin{cases} C\frac{(\log{z})^{1/k}}{z^{(k-1)/k}}, & \mbox{if } z>z_0 \\ 0 & \mbox{otherwise } \end{cases}$$
where $C$ is a determined constant and $z_0$ is the smallest constant such that $p(z_0)\leq 1/2$.  They then showed that this random sequence is a.s. an asymptotic basis of order $k$ and that (2) holds a.s. for large $n$.

A natural finite variant of the above problem concerns the notion of $k$-additive bases. 
\begin{defn} \emph{With $[n]:=\{1,2,\ldots,n\}$, a set $\ca\subseteq[n]\cup\{0\}$ is said to be a \emph{$k$-additive basis} for $[n]$ if each $j\in[n]$ can be written as $j=x_1+x_2+\ldots+x_k$,  $x_i\in\ca,$ $i=1,\ldots,k$.}\end{defn}
\noindent Note that this definition allows for $x_i=x_j$ $i\neq j$.  In \cite{et}, Erd\H os and Tetali showed that in some probability space, almost all infinite sequences $\mathcal S$ satisfy (2) and are asymptotic bases of order $k$.  It is natural to then ask, for finite $\ca,$ how small $\ca$ can be while still being a $k$-additive basis.   For $k=2$, if $\ca$ is a 2-additive basis we must clearly have ${{\vca+1}\choose{2}}\ge n$, so that $\vca\ge\sqrt{2n}(1+o(1))=1.4142(1+o(1))\sqrt{n}$.  
Extensive work, using predominantly analytic techniques, has been done to improve this trivial lower bound to $\vca\ge1.428\sqrt{n}$ (\cite{moser}); $\vca\ge1.445\sqrt{n}$ (\cite{gn}); $\vca\ge1.459\sqrt{n}$ (\cite{yu};, and $\vca\ge1.463\sqrt{n}$ (\cite{ford}).
The best upper bound appears to be $\vca\le\sim1.871\sqrt{n}$ (\cite{mrose}, \cite{hh}).  

In this paper, we will use a probability model in which each integer in $[n]\cup\{0\}$ is chosen to be in $\ca$ with equal (and low) probability $p=p_n$.  We will then give sharp bounds on the probability that the random set $\ca$ is a $k$-additive basis.  It is evident that smaller numbers must be present in an additive basis, since, e.g., the only way to represent 1 in a 2-additive basis is as 1+0, and so in the random model edge effects come into play.  Therefore, the random ensemble is unlikely to form an additive basis unless we adopt a different approach.  This may be done in two ways, which lead to the following alternative definitions: 

\begin{defn}  \emph{ A set $\ca\subseteq[n-1]\cup\{0\}$ is said to be a \emph{ modular $k$-additive basis} for $[n]$ if each $j\in[n-1]\cup\{0\}$ can be written as $j=x_1+x_2+\ldots+x_k\pmod n$,  $x_i\in\ca,$ $i=1,\ldots,k$.}\end{defn}

\begin{defn}  \emph{A set $\ca\subseteq[n]\cup\{0\}$ is said to be a \emph{ truncated $k$-additive basis} for $[n]$ if each $j\in[\alpha n, (k-\alpha)n]$ can be written as $j=x_1+x_2+\ldots+x_k$,  $x_i\in\ca,$ $i=1,\ldots,k; \alpha\in(0,1)$.}\end{defn}

It turns out that definitive results using Definition 1 have been proved in the papers of Yadin \cite{ya} and Sandor \cite {sandor}, and we thus focus on developing results using Definition 2.  Although our model differs from that of Erd\H os and Tetali (our set is constructed using constant probability $p$ as opposed to the $p(z)$ used in \cite{et}), the output threshold probabilities for the size of a truncated (or modular) $k$-additive basis end up being remarkably close to the input probabilities used by Tetali-Erd\H os to construct their asymptotic bases.  We will stress the similarities between our results as appropriate.  It also bears mentioning that the threshold size for our random $\ca$ to be a $2$-additive basis is, up to a logarithmic factor, similar to the bounds outlined above using the analytic techniques.

Our work is organized as follows:  We present threshold results on truncated $2$-additive bases in Section 2, using the Stein-Chen method of Poisson approximation (see \cite{bhj}), which is an alternative to the Janson  inequalities and Brun's sieve used in \cite{ya},\cite{sandor}. This method allows one to not just find the limiting probability that $\ca$ forms an additive basis, but also to approximate the probability distribution $\cl(X)$ of the random variable $X$, defined as the number of integers $j$ that cannot be expressed as a sum of elements in $\ca$. In the modular case, we will, accordingly, show as a corollary how our result partially extend those in \cite{ya}, \cite{sandor}. 
In Section 4, we consider similar questions for truncated $k$-additive bases; at times, the Janson exponential inequalities (\cite{as}, \cite{janson}) are used to estimate some critical baseline quantities that were calculated exactly in Section 2 for $k=2$. 

\begin{rem}  \emph{Throughout the rest of the paper, we suppress the descriptor  ``additive", referring simply to ``truncated $k$-bases," and, occasionally, ``modular $k$-bases". 
}
\end{rem}

\section{$2$-Additive Bases}  
We begin by investigating truncated $2$-bases in our random model outlined above. We will then state the analogous result for modular 2-bases and we will explain how the modular results (and their generalization to arbitrary k-bases) extends the existing work in \cite{ya} and \cite{sandor}.  As with all results in this paper, we seek to find the threshold value at which the bases emerge.   

We begin by selecting each integer from $\{0,1,\ldots, n\}$ independently with probability $p=p_n$.  We thus obtain a random set $\ca$, and denote by $X$ the number of integers in $[\alpha n, (2-\alpha)n]$ that cannot be written in the form $x_1+x_2$, $x_i\in\ca$.  Evidently, $\ca$ is a truncated $2$-basis if and only if $X=0$.  We can write $X=\sum_{j=\lceil\alpha n\rceil}^{\lfloor\maln\rfloor}I_j$, where $I_j$ equals one or zero according as the integer $j$ cannot or can be represented as a 2-sum of elements in $\ca$.  To simplify the notation a bit, we will write simply $\sum_{j=\alpha n}^{\maln}I_j$ for $\sum_{j=\lceil\alpha n\rceil}^{\lfloor\maln\rfloor}I_j$ in the sequel.  The main results of this section are summarized in the following theorem:
\begin{thm}
\label{thm:2trun}
Pick each integer in $[n]\cup\{0\}$ independently with probability $p=p_n\to0$ thus getting a random set $\ca$.  
 
i. Let $Y\sim\P(\lambda=\e(X))$.  Let $\delta>0$, then for all 
\begin{equation}\label{p1}p\ge\sqrt{\frac{(\frac{1}{\alpha}+\delta)\log n}{n}},\end{equation}it hold that
$d_{TV}(X,Y)  \rightarrow 0$  ($n\rightarrow\infty$).

ii.  With 
\begin{equation}\label{p2} p={\sqrt{\frac{\frac{2}{\alpha}\log n-\frac{2}{\alpha}\log \log n+A_n}{n}}}\end{equation}
\[A_n\to\infty\Rightarrow \p(\ca\ {\rm is\ a\ truncated\ 2-basis})\to1\enspace(n\to\infty);\]
\[A_n\to-\infty\Rightarrow \p(\ca\ {\rm is\ a\ truncated\ 2-basis})\to0\enspace(n\to\infty);\] and 
\[A_n\to A\Rightarrow \p(\ca\ {\rm is\ a\ truncated\ 2-basis})\to\exp\lc-{2\alpha e^{-\alpha A/2}}\rc\enspace(n\to\infty)\]
\end{thm}
\begin{proof}

Here and throughout the paper, let ${\cal S}$ denote the random sumset generated by $\ca$.  The first step in the proof will be to calculate the precise asymptotics of $\e(X)$
\begin{prop} With $X$ as above,  $$\e(X)=\sum_{j=\alpha n}^{\maln}\p(I_j=1)\approx2\sum_{j=\aln}^n(1-p^2)^{j/2}\asymp\frac{4}{p^2}\exp\{-np^2\alpha/2\}.$$
\end{prop}
\begin{proof}
 This is a critical computation and needs to be justified for the entire range of $p'$s that we encounter, in particular for all $p$ satisfying \eqref{p1}.  The first direction is easy:
\begin{eqnarray*}
2\sum_{j=\aln}^n(1-p^2)^{j/2}&\le&2\sum_{j=\aln}^\infty(1-p^2)^{j/2}\\
&\le& \frac{2(1-p^2)^{\aln/2}}{1-\sqrt{1-p^2}}\\
&\le&\frac{4}{p^2}(1-p^2)^{\aln/2}\\
&\le&\frac{4}{p^2}\exp\{-np^2\alpha/2\},
\end{eqnarray*}
using the inequalities $1-x\le e^{-x}$ and $\sqrt{1-p^2}\le 1-\frac{p^2}{2}$.  For the other direction, we have
\begin{eqnarray*}
2\sum_{j=\aln}^n(1-p^2)^{j/2}&\ge&\frac{2(1-p^2)^{\aln/2}-2(1-p^2)^{n/2}}{1-\sqrt{1-p^2}}\\
&=&\frac{2(1-p^2)^{\aln/2}(1-(1-p^2)^{\maln/2})}{1-\sqrt{1-p^2}}\\
&\ge&\frac{2(1-p^2)^{\aln/2}\lr1-e^{-np^2(1-\alpha)/2}\rr}{\frac{p^2}{2}+\frac{p^4}{8}}\\
&\asymp&\frac{4}{p^2}(1-p^2)^{\aln/2}\\
&\asymp&\frac{4}{p^2}\exp\{-np^2\alpha/2\}.
\end{eqnarray*}
\end{proof}
With $p$ defined as in \eqref{p2} we see
that $\e(X)=(1+o(1))(2\alpha\cdot\exp\{-\alpha A_n/2\})$.  As is often the case with threshold phenomena, Markov's inequality can be used to easily establish the first part of ii, as we have
$\p(X\ge1)\le\e(X)\to0$
provided that $A_n\to\infty$ in \eqref{p2}.  

To establish the second and third statements in part ii.~of the theorem, we go beyond estimating the point probability $\p(X=0)$, using instead the Stein-Chen method of Poisson approximation \cite{bhj} to establish a total variation approximation for the distribution $\cl(X)$ of $X$.  If the distribution of $X$ is approximately Poisson, then we will have $e^{-\lambda}-\ep_n\le\p(X=0)\le e^{-\lambda}+\ep_n$, where $\lambda=\e(X)$ is the mean of the approximating Poisson variable and $\ep_n$ is the total variation error bound for the approximation.   We have that $\e(X)\to\infty$ if $A_n\to-\infty$; $\e(X)\to0$ if $A_n\to\infty$; and $\e(X)\to 2\alpha e^{-\alpha A/2}$ if $A_n\to A$.  Thus, if we can show that $\ep_n$ tends to zero in a window around $$p={\sqrt{\frac{\frac{2}{\alpha} \log n}{n}}},$$ we will have that $\p(X=0)\to0$ if  
$A_n\to-\infty$; $\p(X=0)\to1 $ if $A_n\to\infty$; $\p(X=0)\to\exp\{-2\alpha e^{-\alpha A/2}\}$ if $A_n\to A$.  

With $Y\sim\P(\lambda=\e(X))$, and throughout writing $\tv(A,B)$ instead of the more appropriate $\tv(\cl(A),\cl(B))$, we seek to bound $$d_{TV}(X,Y):=\sup_{A\subseteq{\mathbb Z}^+}\vert\p(X\in A)-\p(Y\in A)\vert.$$ 
Following \cite{bhj}, we first need to determine, for each $j$ separately, an auxiliary sequence  of variables $J_{j_i}$ defined on the same probability space with the property that 
\begin{equation}\label{dist}\cl(J_{j_1},J_{j_2},\ldots)=\cl(I_1,I_2,\ldots|I_j=1).\end{equation}

Our explicitly constructed coupling of the $J_{j_i}$'s is as follows:
If $I_j=1$, set $J_{j_i}=I_i$ for all $i\neq j$.
If $I_j=0$, for all pairs $x_1,x_2\in \ca $ such that $x_1\neq x_2$ and $x_1+x_2=j,$ remove $x_1$ from $\ca$ with probability $\frac{p(1-p)}{1-p^2},$ remove $x_2$ from $\ca$ with probability  $\frac{p(1-p)}{1-p^2},$
and remove both $x_1$ and $x_2$ from $\ca$ with probability $\frac{p^2+1-2p}{1-p^2}$. If $x\in\ca$ and $x+x=i,$ then we remove $x$ from $\ca$ with probability 1.  Finally, 
define  $J_{j_i}=1$ if $
i\not\in {\cal S}$ { after the above coupling is implemented.}  It is clear that \eqref{dist} is satisfied since we have ``de-selected" offending integers based on the conditional probability of one or both integers in a pair being absent, given that both are not present.  The total variation bounds derived in \cite{bhj} are expressed in terms of the probability that the coupled indicator variables are different after the coupling is implemented; i.e. $I_i=1,J_{j_i}=0$ or $I_i=0,J_{j_i}=1$.
Now, $\p(I_i=1,J_{j_i}=0)=0$, since if integer $i$ is not present in $\mathcal S$, it cannot magically appear after some integers have been de-selected.

The formula we need thus reduces to
$$d_{TV}(X,Y)\leq \frac{1-e^{-\lambda}}{\lambda}\sum_j\bigg(\p^2(I_j=1)+\p(I_j=1)\cdot\sum_{i\neq j}\p(I_i=0, J_{j_i}=1)\bigg).$$
There are two terms above.  The first,
  \[\frac{1-e^{-\lambda}}{\lambda}\sum_j\p^2(I_j=1)\leq \frac{1}{\lambda}\max_i\p(I_i=1)\sum_i\p(I_j=1)=\max_i\p(I_i=1),\]
and thus is bounded above by
$\p(I_{\lceil\aln\rceil}=1)\approx(1-p^2)^{\aln/2}$.  This bound tends to zero if $1/\sqrt{n}=o(p)$, a condition that is satisfied for $p$ satisfying \eqref{p1}.
  
To bound the second term in the sum, we begin by noting
$$\frac{1-e^{-\lambda}}{\lambda}\sum_j\left[\p(I_j=1)\sum_{i\neq j}\p(I_i=0, J_{j_i}=1)\right]\le\max_j\sum_{i\neq j}\p(I_i=0, J_{j_i}=1).$$
We further bound this second term by conditioning on the number of 2-sums of $i$ that are present pre-coupling, denoting this number by $B_i$.  Denote the sumset of $\ca$ post-coupling by ${\cal S}^*$.  We see that 
$$\p(I_i=0, J_{j_i}=1)=\sum_{k=1}^{\lceil \frac{i+1}{2}\rceil} \p(i\not\in {\cal S}^*\vert B_i=k)\cdot \p(B_i=k).$$  Note that if there exists $x$ such that $i=2x$, then 
$\p(B_i=k)={{\lceil \frac{i+1}{2}\rceil-1}\choose{k}}p^{2k}(1-p^2)^{\lceil \frac{i+1}{2}\rceil-k-1}(1-p)+ {{\lceil \frac{i+1}{2}\rceil-1}\choose{k-1}}p^{2k-2}(1-p^2)^{\lceil \frac{i+1}{2}\rceil-k}p,$ otherwise (when $i$ is odd and $f(i):=\lceil \frac{i+1}{2}\rceil$ to simplify the notation)
\begin{align*}
&\p(B_i=k)={{f(i)}\choose{k}}p^{2k}(1-p^2)^{f(i)-k}\\
&={{f(i)-1}\choose{k}}p^{2k}(1-p^2)^{f(i)-k}+ {{f(i)-1}\choose{k-1}}p^{2k}(1-p^2)^{f(i)-k}\\
&\leq{{f(i)-1}\choose{k}}p^{2k}(1-p^2)^{f(i)-k}+ {{f(i)-1}\choose{k-1}}p^{2k-2}(1-p^2)^{f(i)-k}p\\
&\asymp {{f(i)-1}\choose{k}}p^{2k}(1-p^2)^{f(i)-1-k}(1-p)+ {{f(i)-1}\choose{k-1}}p^{2k-2}(1-p^2)^{f(i)-k}p
\end{align*}
Our next step is to bound $\p(i\not\in {\cal S}^*\vert B_i=k)$ via 
$$\p(i\not\in {\cal S}^*\vert B_i=k)\leq \left( \p(\text{a given 2-sum of }i\text{ removed}\vert B_i=k)\right)^k.$$  To bound the above term, we will assume all elements of the 2-sums of $i$ are part of 2-sums of $j$ and thus have positive probability of being removed by the coupling process.  Fix $x,y$ $x\neq y$, such that $x+y=i$.  For ease of notation, define $P_x$ to be the event that $x$ is part of a pre-coupling 2-sum of $j$, and analogously define $P_y$.  Further define $R_x$ to be the event $x$ is removed from $\ca$ by the coupling, and analogously define $R_y$.  Consider the following simple calculations: 
$$\p(P_x,\ P_y^c,\ R_x)=p(1-p)\left(1-\frac{p(1-p)}{1-p^2}\right)\le p,$$
\n since the only undesirable outcome is if only the ``other" component of the 2-sum of 0 is removed in lieu of $x$;
$$\p(P_x,\ P_y,\ R_x,\ R_y^c)=p^2\frac{p(1-p)}{1-p^2}\lr1-\frac{p(1-p)}{1-p^2}\rr\le p^3;$$
\n and
$$\p(P_x,\ P_y,\ R_x,\ R_y)=p^2\lr1-\frac{p(1-p)}{1-p^2}\rr^2\le p^2.$$
In the case where $x=y$, the only case we would need to consider is 
$$\p(P_xR_x)=p.$$
From this, we see that $$\p(\text{a given 2-sum of }i \text{ removed}\ |\ B_i=k)\leq 2p+2p^3+p^2.$$ Let $C(p)=2p+2p^3+p^2$.  Our above calculations yield that $\max_j\sum_{i\neq j}\p(I_i=0, J_{j_i}=1) \leq\Sigma_1+\Sigma_2$, where
$$\Sigma_1:=\max_j \sum_{i\neq j}\sum_{k=1}^{\lceil \frac{i+1}{2}\rceil}{{\lceil \frac{i+1}{2}\rceil-1}\choose{k}}p^{2k}(1-p^2)^{\lceil \frac{i+1}{2}\rceil-k-1}(1-p)(C(p))^k$$ and
$$\Sigma_2:=\max_j \sum_{i\neq j}\sum_{k=1}^{\lceil \frac{i+1}{2}\rceil}{{\lceil \frac{i+1}{2}\rceil-1}\choose{k-1}}p^{2k-2}(1-p^2)^{\lceil \frac{i+1}{2}\rceil-k}p(C(p))^k.$$
Now, 
\begin{align*}
\Sigma_1&=\max_j \sum_{i\neq j}(1-p)\left[(1-p^2+p^2C(p))^{\lceil \frac{i+1}{2}\rceil-1}-(1-p^2)^{\lceil \frac{i+1}{2}\rceil-1}\right]\\
&\leq (1-p)\max_j\sum_{i\neq j}\bigg(\left\lceil \frac{i+1}{2}\right\rceil -1\bigg)C(p)p^2(1-p^2+C(p)p^2)^{\lceil \frac{i+1}{2}\rceil-2}\\
&\asymp \max_j\sum_{i\neq j}\ic C(p)p^2e^{-\ic p^2}\\
&\leq n\cdot\frac{(2-\alpha )n}{2}C(p)p^2e^{-\aln p^2/2}\\
&=O\left(n^2p^3e^{-\aln p^2/2}\right)
\end{align*}
and hence $\Sigma_1\rightarrow 0$ if $p$ satisfies \eqref{p1}.  For $\Sigma_2$, we have
\begin{align*}
\Sigma_2&=C(p) p \max_j \sum_{i\neq j}\sum_{k=1}^{\lceil \frac{i+1}{2}\rceil}{{\lceil \frac{i+1}{2}\rceil-1}\choose{k-1}}p^{2k-2}(1-p^2)^{\lceil \frac{i+1}{2}\rceil-k}(C(p))^{k-1}\\
&=C(p) p \max_j \sum_{i\neq j} (1-p^2+C(p)p^2)^{\lceil \frac{i+1}{2}\rceil-1}\\
&\asymp C(p) p \max_j \sum_{i\neq j}e^{-\ic p^2}\\
&\leq C(p) p ne^{-\aln p^2/2}\\
&=o(\Sigma_1)
\end{align*}
and we have that $d_{TV}(X,Y)\rightarrow 0$ if $p$ satisfies \eqref{p1},
and for $p$ in this range
$$|\p(X=0)-e^{-\lambda}|\rightarrow 0$$
 which finishes the proof.  Note that if we consider a $p$ not in this range, the result will hold by monotonicity.
\end{proof}
The modular 2-basis result is proved similarly, and suppressing all details we state the result:
\begin{thm}
\label{thm:old}
Pick each integer in $[n-1]\cup\{0\}$ independently with probability $p=p_n$ thus getting a random set $\ca$.  
Then,

i.  Let $X$ denote the number of missing integers in the modulo-2 sumset of $\ca$, and let $Y\sim\P(\lambda=\e(X))$.  Let $\delta>0$, then for all 
$$p\ge\sqrt{\frac{(1+\delta)\log n}{n}}$$
it holds that $d_{TV}(X,Y)\rightarrow 0$.

ii.  If we let
\[p={\sqrt{\frac{2 \log n+A_n}{n}}} \text{ where $\vert A_n\vert=o(\log n)$,} \] then
\[A_n\to\infty\Rightarrow \p(\ca\ {\rm is\ a\ modular\ 2-basis})\to1\enspace(n\to\infty);\]
\[A_n\to-\infty\Rightarrow \p(\ca\ {\rm is\ a\ modular\ 2-basis})\to0\enspace(n\to\infty);\]
\[A_n\to A\in{\mathbb R}\Rightarrow \p(\ca\ {\rm is\ a\ modular\ 2-basis})\to\exp\lc-{e^{-A/2}}\rc\enspace(n\to\infty).\]
\end{thm}

The second part of Theorem \ref{thm:old} was proven in \cite{ya} using Janson's correlation inequalities and in  \cite{sandor} using the method of Brun's sieve (see for example \cite{as}).  In \cite{sandor}, the author is able to derive  part i.~of Theorem \ref{thm:old} at the threshold value of $p$, while the Stein-Chen method allows us to derive the result in a window about the threshold, thus somewhat generalizing the previous known results.

\begin{rem}
\emph{
We have so far dealt with random sets of fixed expected size, and now indicate briefly how we can easily transition to the case of random sets of fixed size.  This can be done for all the results in this paper, but we indicate the method in the context of Theorem 5:  Choose one of ${{n+1}\choose{\vca}}$ families randomly, and suppose $\vca=\sqrt{Kn\log n}; K>\frac{2}{\alpha}$ (this corresponds to the expected size of $\ca$ with $p=\sqrt{K\log n/n}$).
Then we reconcile the two models as follows with $p=\sqrt{K\log n/n}$:
\begin{eqnarray*}
&&{}\p(\ca\ {\rm is\ not\ a\ basis}\vert \vca=\sqrt{Kn\log n})\\
&&{}\le\p(\ca\ {\rm is\ not\ a\ basis}\vert \vca\le\sqrt{Kn\log n})\\
&&{}\le\p(\ca\ {\rm is\ not\ a\ basis})/\p(\vca\le\sqrt{Kn\log n})\\
&&{}\approx 2\p(\ca\ {\rm is\ not\ a\ basis})\to0,
\end{eqnarray*}
by Theorem 5 and the central limit theorem.  A similar argument holds if $K<\frac{2}{\alpha}$, or even if $\vca=\sqrt{\frac{2}{\alpha}n\log n-\frac{2}{\alpha}n\log\log n+nA_n}$, where $\vert A_n\vert\to\infty$.  It follows that
we can easily go back and forth from the independent model to the fixed set size model except possibly when we are at the threshold, i.e., when
\[p={\sqrt{\frac{\frac{2}{\alpha}\log n-\frac{2}{\alpha}\log \log n+A_n}{n}}},\]
with $A_n\to A\in{\mathbb R}$.}
\end{rem}

\section{Modular $k$-Additive Bases}  Let us briefly  now turn our attention to modular $k$-additive bases.  We define $\ca$ as in Section 2, and will choose integers to be in $\ca$ with probability $p=p_n$.  Define
$$\mathcal{S}:=\{x_1+x_2+\ldots+x_k \ \pmod n | x_i\in \ca\}.$$  We define $I_j$ to be 0 if $j\in \mathcal{S}$ and 1 if $j\notin \mathcal{S}$.  As before, let $X=\sum_{j=0}^{n-1}I_j$.  In \cite{sandor} it was proven, again using the method of Brun's sieve, that
 \begin{thm}
With $X$ defined as above, if we choose elements of $[0,n-1]$ to be in $\ca$ with probability
$$p=\sqrt[k]{\frac{k!\log n+A_n}{n^{k-1}}},$$ 
where $\vert A_n\vert=o(\log n)$, then 
$$\p(\ca\ is\ a\ modular\ k-basis)\rightarrow\begin{cases} 0 & \mbox{if}\ A_n\rightarrow-\infty\\  1 & \mbox{if}\ A_n\rightarrow \infty \\ \exp\{- e^{-A/k!}\} & \mbox{if}\ A_n\rightarrow A\in{\mathbb R}\end{cases}.$$ 
Also, if $Y\sim\P(\e(X))$,  
$$d_{TV}(X,Y)\rightarrow 0 \mbox{ as } n\rightarrow \infty.$$ .
\end{thm}

Using the Stein-Chen method we are able to reproduce the above result and prove the total variation convergence for all $p\in[p_0,1]$,where $p_0>\sqrt[k]{\frac{Ak!\log n}{n^{k-1}}}$, with $A>(k-1)/k$.  We will provide proof for our contribution to the window of convergence, omitting proofs when results were derived also in \cite{sandor}

With $X$ as above, we first derive (as in \cite{sandor}) that
$$\e(X)= n \cdot \exp\left\{-\frac{n^{k-1}p^k}{k!}(1+O(1/np))\right\}.$$ 

\noindent Following \cite{bhj}, we first need to determine, for each $j$ separately, an auxiliary sequence  of variables $J_{j_i}$ defined on the same probability space with the property that 
\begin{equation}
\label{eq:coup}
\cl(J_{j_1},J_{j_2},\ldots)=\cl(I_1,I_2,\ldots|I_j=1).\end{equation}
Such a coupling can {\it probably} be described explicitly as we did for $k=2$ but there is no need to do so:  It is clear that the more integers in $\ca$, the higher the probability that $I_j$ is 0, as integer $j$ is more likely to be representable as a $k$-sum.  Therefore, the $I_j$'s are decreasing functions of the baseline i.i.d.~random variables $\{Y_i\}$ that have distribution $\p(Y_0=1)=p; \p(Y_0=0)=1-p$, and a monotone coupling satisfying (\ref{eq:coup}) exists.  Thus 
with $\lambda:=\e(X)$ and 
$Y\sim$\P($\lambda$), we can apply Theorem 2.E and Corollary 
2.C.4 in \cite{bhj} to get 
\begin{align*}
d_{TV}(X,Y)&\leq \frac{1-e^{-\lambda}}{\lambda}\left(\v(X)-\lambda+2\sum_{j=0}^{n-1}\p^2(I_j=1)\right)\\
&\leq \frac{\v(X)}{\lambda}-1+2\p(I_0=1)\\
\end{align*}

Now, from the asymptotic formula of $\e(X)$, we have that 
$2\p(I_0=1)\asymp 2\text{exp}\left\{-\frac{n^{k-1}p^k}{k!}\right\}$ and this goes to 0 as $n\rightarrow\infty$ if
$$p=\frac{k!^{1/k}\phi(n)}{n^{(k-1)/k}}$$
where $\phi(n)\rightarrow\infty$ as $n\rightarrow\infty$ is arbitrary.

We next need to show that $\frac{\v(X)}{\lambda}=1+o(1)$.
\begin{align*}
\v(X)&=\v\left(\sum_{j=0}^{n-1}I_j\right)\\
&=\sum_{j=0}^{n-1} \v(I_j)+2\sum_{i<j}\left(\e(I_iI_j)-\e(I_i)\e(I_j)\right)\\
&=\sum_{j=0}^{n-1}\e(I_j)-\sum_{j=0}^{n-1}\e(I_j)^2+2\sum_{i<j}\e(I_iI_j)-n(n-1)\e(I_0)\e(I_1)\\
&= \lambda-(n+n(n-1)) \e^2(I_0)+2\sum_{i<j}\e(I_iI_j),
\end{align*}
so that
\begin{equation}
\label{eq:var}
\frac{\v(X)}{\lambda}-1=\frac{(n-1)\e(I_0I_1)}{\e(I_0)}-n\e(I_0).\end{equation}
Next, we need to bound the growth of $\e(I_iI_j)$ for $i<j$.  Now $\e(I_iI_j)$ equals the probability then neither $i$ nor $j$ is in $\cs$, and thus $\e(I_iI_j)=\p(S=0)$, where
$$S=\sum_{r=i,j}\sum_{l=1}^k\sum_{{{\bf a}:=\{a_1,\ldots,a_l\}:{\sum a_i=r}}}J_{r,l,{\bf a}},$$ 
where 
$J_{\bf a}:=J_{r,l,{\bf a}}$ equals one if the $l$ integers in {\bf a} that sum to $r$ (with possible repetition) are all selected to be in $\ca$.  We have that (see for example Lemma 1.5 in \cite{sandor})
\[\e(S)=2\frac{{n \choose k}}{n}p^k(1+O(1/np)).\]  Set $\{r,l,{\bf a}\}\sim\{s,m,{\bf b}\}$ if ${\bf a}\cap{\bf b}\ne\emptyset; \{r,l,{\bf a}\}\ne\{s,m,{\bf b}\}$.  For the associated dependency graph, we see that
\[\Delta=\sum_{\{r,l,{\bf a}\}}\sum_{\{s,m,{\bf b}\}\sim\{r,l,{\bf a}\}}\e(J_{\bf a}J_{\bf b})=O(n^{2k-3}p^{2k-1})=O(n^{k-2}p^{k-1}\e(S)),\]
and thus by Janson's inequality we have
\[\e(I_0I_1)\le e^{-\e(S)+\Delta}=\exp\{-2\frac{{n \choose k}}{n}p^k(1+O(1/np)+O(n^{k-2}p^{k-1})\}.\]
Returning to (\ref{eq:var}), we see that
\begin{eqnarray*}&&\frac{\v(X)}{\lambda}-1\\&\le&n\frac{\lr\exp\{-2\frac{{n \choose k}}{n}p^k(1+O(1/np)+O(n^{k-2}p^{k-1})\}-\exp\{-2\frac{{n\choose k}}{n}p^k(1+O(1/np))\}\rr}{\exp\{-\frac{{n \choose k}}{n}p^k(1+O(1/np))\}}\\
&=&n\frac{\exp\{-2\frac{{n\choose k}}{n}p^k(1+O(1/np))\}}{\exp\{-\frac{{n \choose k}}{n}p^k(1+O(1/np))\}}\lr\exp\{-2\frac{{n\choose k}}{n}p^kO(n^{k-2}p^{k-1})\}-1\rr\\
&\sim&\lambda\cdot n^{2k-3}p^{2k-1}.
\end{eqnarray*}
Thus, the total variation distance goes to 0 if 
\[n^{2k-2}p^{2k-1}\exp\{-n^{k-1}p^k/k!\}\to0, \] which holds if $p>\frac{(Ak!\log n)^{1/k}}{n^{(k-1)/k}}, A>\frac{k-1}{k}$. Poisson approximation is thus valid in a window that includes the threshold value of $p=\frac{(k!\log n)^{1/k}}{n^{(k-1)/k}}$.  

\section{Truncated $k$-bases}

The analysis for truncated bases is more subtle than that for modular bases, since edge effects need to be handled carefully.  Following the same setup as before, we wish to represent each $j\in[\a n, (k-\a)n]$ as $j=x_1+\ldots+x_k$ for $x_i\in \ca$, where $x_1\le x_2\le\ldots\le x_k$.  Again, for each $j$, let $I_j$ equal 1 if $j$ cannot be expressed as a $k$-sum of elements of $\ca$ and 0 otherwise.  We set $X:=\sum_{j=\a n}^{(k-\a)n}I_j$, and note that $X=0\Leftrightarrow\ca$ is a  truncated $k$-basis.
\begin{thm}
\label{thm:ktrun}
With $X$ defined as above, if we choose elements of $\{0\}\cup[n]$ to be in $\ca$ with probability
$$p=\sqrt[k]{\frac{K\log n-K \log{\log{n}}+A_n}{n^{k-1}}},$$ 
where $\vert A_n\vert=o(\log{\log{n}})$ and $K=K_{\a,k}=\frac{k!(k-1)!}{\a^{k-1}}$, then 
$$\p(\ca\ is\ a\ truncated\ k-basis)\rightarrow\begin{cases} 0 & \mbox{if}\ A_n\rightarrow-\infty\\  1 & \mbox{if}\ A_n\rightarrow \infty \\ \exp\{- \frac{2\alpha}{k-1}e^{-A/K}\} & \mbox{if}\ A_n\rightarrow A\in{\mathbb R}\end{cases}.$$
Also for $$p=\sqrt[k]{\frac{\beta K\log n}{n^{k-1}}},$$ $\beta>(k-1)/k$, and $Y\sim\P(\e(X))$,  
$$d_{TV}(X,Y)\rightarrow 0 \mbox{ as } n\rightarrow \infty.$$
\end{thm}
Before proving this theorem, we need some preliminary work.  Let $S_j$ be the set of all unordered $k$-tuples of nonnegative integers in $\{0\}\cup[n]$ that sum to $j$.

\begin{cm}
For $j\in[\a n,n]$, $$|S_j|=j^{k-1}/[(k-1)!k!]+O(j^{k-2}).$$
\end{cm}
\begin{proof}
The number of ordered $k$-tuples of nonnegative integers that sum to $j$ is $\binom{j+k-1}{j}\sim\frac{j^{k-1}}{(k-1)!}$.  All such tuples will be composed entirely of numbers in $\{0\}\cup[j]$, and at most ${k \choose 2}\cdot n\cdot {{j+k-3}\choose{j}}=O(j^{k-2})$ of these contain a number repeated once or more often.  We can disregard these in the asymptotic analysis, and consider the remaining unordered and ordered tuples.  Each remaining unordered tuple appears $k!$ times among the remaining ordered tuples, giving us the desired first order asymptotics.
\end{proof}

\begin{cm}
If $1\le j\le kn/2$, then
$|S_j|=|S_{kn-j}|$
\end{cm}
\begin{proof}
There is a bijection between $k$-tuples in $S_j$ for $1\le j\le kn/2$ and those in $S_j$ for $j\ge kn/2$ given by $\{x_1,x_2,\ldots,x_k\}\leftrightarrow\{n-x_1,n-x_2,\ldots,n-x_k\}$.
\end{proof}
Lastly, we have:
\begin{cm}
For $j\in[n+1,(k-1)n]$, 
$$|S_j|\ge \frac{n^{k-1}}{(k-1)!k!}+O(n^{k-2}).$$
\end{cm}
\begin{proof}
We will use the fact that $|S_j|$, i.e. the number of partitions of $j$ of size at most $k$ each part of which is less than or equal to $n$, is the coefficient of $q^j$ in the $q$-binomial coefficient 
$$\binom{n+k}{k}_q:=\frac{(1-q)(1-q^2)\cdots(1-q^{n+k})}{(1-q)(1-q^2)\cdots(1-q^{n})(1-q)(1-q^2)\cdots(1-q^{k})}.$$
It is well known, see for example \cite{oh}, that $\binom{n+k}{k}_q=\sum a_i q^i$ is a polynomial in $q$ and that the coefficients are unimodal, namely $a_{j-1}<a_j$ for $j\leq nk/2$.  Claim 11 yields that
$$a_n=|S_n|=\frac{n^{k-1}}{(k-1)!k!}+O(n^{k-2}),$$
and the proof now follows directly from Claim 12.
\end{proof}

We next begin our analysis of $\e(X)$ with a preliminary claim.
\begin{cm}
With $S_j$ defined as above, 
$$\e(X)= \sum_{j=\a n}^{(k-\a)n}\exp\left(-|T_j|p^k(1+o(1))\right),$$
where $T_j$, the set of $k$-tuples of distinct elements that add to $j$ satisfies
$$|T_j|\begin{cases} =j^{k-1}/(k-1)!k!, & \mbox{if } j\in[\a n,n] \\ \ge {n^{k-1}}/{(k-1)!k!} & \mbox{if } j\in[n,(k-1)n]\\
=(kn-j)^{k-1}/(k-1)!k! & \mbox{if } j\in[(k-1)n,(k-\a)n] \end{cases}$$
\end{cm}
\begin{proof}
Fix $j$, and let $\{B_{(j,i)}\}_{i=1}^{|S_j|}$ be the event that the $i$th unordered $k$-tuple of elements of $\{0\}\cup[n]$ that add to $j$ is in $\ca$.  
Then by Janson's inequality, we have
\begin{align*}
\p(I_j=1)&=\p(\cap_{i} B_{(j,i)}^C)\\
&\geq (1-p)^{O(1)}(1-p^2)^{O(j)}(1-p^3)^{O(j^2)}\ldots(1-p^{k-1})^{O(j^{k-2})}(1-p^k)^{|T_j|}\\
&\geq \exp(-\frac{O(1)p}{1-p}-\frac{O(j)p^2}{1-p^2}-\ldots\frac{|T_j|p^k}{1-p^k})\\
&= \exp\bigg(-|T_j|p^k\left(\frac{1}{1-p^k}+\frac{O(j^{k-2})}{(1-p^{k-1})|T_j|p}+\ldots+\frac{O(1)}{(1-p)|T_j|p^{k-1}}\right)\bigg)\\
&=\exp \lbrace -|T_j|p^{k}(1+o(1))\rbrace \\
\end{align*}
as for $x\in(0,1)$, $1-x\geq e^{-x/(1-x)}$, and with our choice of $p$, $np\rightarrow\infty$ and 
for all $j$, we have $j=O(n)$.

For an upper bound, Janson's inequality yields 
$$P(I_j=1)\leq \exp\left(-|T_j|p^{k-1}(1+o(1))+\Delta_j/2\right).$$
Now $\Delta_j=O(n^{2k-3}p^{2k-1})$ since the maximal contribution to $\Delta_j$ is when the two tuples have distinct elements and intersect in a single element; as for all values of $p$ in our window, $np\to\infty$.  Thus 
\begin{align*}
P(I_j=1)&\leq \exp\left(-|T_j|p^{k-1}+O(n^{2k-3}p^{2k-1})\right)\\
&=\exp\left(-|T_j|p^{k}(1+O(n^{k-2}p^{k-1}))\right).
\end{align*}
With $p=o(n^{-(k-2)/(k-1)})$, a condition satisfied by all $p's$ in our theorem, $n^{k-2}p^{k-1}\rightarrow 0$, finishing the proof. 
\end{proof}
We are now ready to begin the proof of Theorem \ref{thm:ktrun}. 
\begin{proof}
We now have that (with $(1+o(1))$ terms suppressed
$$\e(X)= 2\sum_{j=\a n}^n \exp\{-p^kj^{k-1}/{(k-1)}!k!\}+\sum_{j=n}^{(k-1)n} \exp\{-p^k|T_j|\}.$$
The first summation, which we will call $\Sigma_1$ (correspondingly calling the second sum $\Sigma_2$) can be bounded above as follows, recalling the definition of $K$ from the statement of the theorem, and setting $B=\a^{k-2}/k!(k-2)!$:
\begin{align*}
\Sigma_1&\leq 2\sum_{j=\a n}^{\infty} \exp\{-p^kj^{k-1}/(k-1)!k!\}\\
&\leq 2 \exp\{-n^{k-1}p^k/K\}\left(\sum_{j=0}^{\infty} \lr\exp\{-(k-1)(n\a)^{k-2}p^k/(k-1)!k!\}\rr^j\right)\\
&=\frac{2 \exp\{-n^{k-1}p^k/K\}}{1-\exp\{-Bn^{k-2}p^k\}}\\
&\leq \frac{2 \exp\{-n^{k-1}p^k/K\}}{Bn^{k-2}p^k}(1+Bn^{k-2}p^k)\\
&=\frac{2 \exp\{-n^{k-1}p^k/K\}}{Bn^{k-2}p^k}(1+o(1))
\end{align*}
where we used the facts that $\frac{x}{x+1}\leq 1-e^{-x}$ in the fourth line and, in the final line of the display, that for all choices of $p$ considered in the theorem, $n^{k-2}p^k\rightarrow 0$.  Finally, the geometric bound in the second line follows from the fact that the ratio of consecutive terms in the sum satisfies
\[\exp\{-\frac{p^k}{(k-1)!k!}[(j+1)^{k-1}-j^{k-1}]\}\ge \exp\{-\frac{p^k}{(k-1)!k!}(k-1)(\a n)^{k-2}\}\]

For $\Sigma_2$ we have
\begin{eqnarray*}
\Sigma_2&=&\sum_{j=n}^{(k-1)n}\exp\{-p^k|T_j|\}\\
&\le&(k-2)n\exp\{-p^kn^{k-1}/(k-1)!k!\}\to0
\end{eqnarray*} 
if $$p\ge\lr\frac{D\log n}{n^{k-1}}\rr^{1/k}; D>(k-1)!k!$$

A lower bound on $\Sigma_1$ (and thus on $\e(X)$) is obtained by using elementary integration by parts to derive a tight estimate for the integral
\[\int_x^\infty e^{-t^{k-1}}dt; x\to\infty,\]
in much the same way that Gaussian tails are analyzed (which is the $k=3$ case.)  We have 
\begin{eqnarray}
\e(X)&\ge&2\sum_{j=\a n}^n \exp\{-p^kj^{k-1}/{(k-1)}!k!\}\nonumber\\
&\sim&2\int_{\a n}^n \exp\{-p^kx^{k-1}/{(k-1)}!k!\}dx\nonumber\\
&=&2\int_{\a n}^\infty \exp\{-p^kx^{k-1}/{(k-1)}!k!\}dx\nonumber\\
&\quad&-2\int_n^\infty \exp\{-p^kx^{k-1}/{(k-1)}!k!\}dx.
\end{eqnarray}
Setting, for $t>0$, $\Psi(t,k)=\int_t^\infty\exp\{-Cx^{k-1}\}dx$, we see that
\begin{eqnarray}\label{eq:gauss}\Psi(t,k)&\le&\frac{1}{t^{k-2}}\int_t^\infty x^{k-2}\exp\{-Cx^{k-1}\}dx\nonumber\\
&=&\frac{1}{C(k-1)t^{k-2}}\int_{Ct^{k-1}}^\infty e^{-u}du\nonumber\\
&=&\frac{1}{C(k-1)t^{k-2}}\exp\{-Ct^{k-1}\},
\end{eqnarray}
and, for another two constants $E, E'>0$,
\begin{eqnarray}
\Psi(t,k)&=&\int_t^\infty \frac{x^{k-2}}{x^{k-2}}\exp\{-Cx^{k-1}\}dx\nonumber\\
&=&\frac{1}{C(k-1)t^{k-2}}\exp\{-Ct^{k-1}\}-\frac{E}{C}\int_t^\infty\frac{1}{x^{k-1}}\exp\{-Cx^{k-1}\}dx\nonumber\\
&\ge&\frac{1}{C(k-1)t^{k-2}}\exp\{-Ct^{k-1}\}-\frac{E'}{C^2}\frac{1}{t^{2k-3}}\exp\{-Ct^{k-1}\}
\end{eqnarray}
using (\ref{eq:gauss}).
Thus, since in our context $C^2t^{2k-3}\gg Ct^{k-2}$, $$\Sigma_1\asymp \frac{2(k-1)!k!}{(k-1)\a^{k-2}n^{k-2}p^k}\exp\{-\a^{k-1}n^{k-1}p^k/(k-1)!k!\}.$$
It is easy to verify that $\Sigma_2=o(1)\Sigma_1$, and thus $$\e(X)\sim\Sigma_1\asymp\frac{2(k-2)!k!}{\a^{k-2}n^{k-2}p^k}\exp\{-\a^{k-1}n^{k-1}p^k/(k-1)!k!\} .$$

Let $\lambda=\e(X)$ and let $Y\sim \P(\lambda)$.  First we note that $\e(X)$ tends to zero, infinity, or $\frac{2\alpha}{k-1}e^{-A/K}$ if $p$ is as stated as in the theorem with $A_n\to\infty$, $A_n\to-\infty$ or $A_n\to A$ respectively.  To complete the proof of Theorem \ref{thm:ktrun}, we use Poisson approximation to show that the total variation distance between $X$ and $Y$ converges to 0, so that, in particular, $\p(X=0)\to e^{-\lambda}$.  

Using Theorem 2.C of \cite{bhj} with $\Gamma_{\alpha}^0=\Gamma_{\alpha}^{-}=\emptyset$,
$$d_{TV}(X,Y)\leq\frac{1-e^{-\lambda}}{\lambda}\left(\sum_j\p^2(I_j=1)+\sum_j\sum_\ell (\e(I_jI_\ell)-\e(I_j)\e(I_\ell))\right);$$ the above is just a variation of the bound used in the proof of Theorem \ref{thm:2trun}.
Now 
\begin{align*}
\frac{1-e^{-\lambda}}{\lambda}\sum_j\p^2(I_j=1)&\leq \mbox{max}_j\p(I_j=1)\\
&\leq \mbox{max}_j\exp\{-|T_j|p^k(1+o(1))\}\\
&\le e^{-\delta n^{k-1}p^k(1+o(1))},
\end{align*}
where $\delta$ is a constant not depending on $p$ or $n$, and this term converges to 0 if $p=\sqrt[k]{\frac{G\log n}{n^{k-1}}}$ for any constant $G>0$.  For the double sum, we start by using the estimates
\[\p(I_i=1)\ge\exp\{-\vert S_i\vert p^k(1+O(1/np))\}\] for $i=j,\ell$, and for the cross product term $\e(I_jI_\ell)$ we use Janson's inequality as in our Stein-Chen treatment of Theorem 9.
For fixed $j,\ell$, note that  $\p(I_jI_\ell=1)=\p(S=0)$, where 
$$S=\sum_{r=j,\ell}\sum_{s=1}^k\sum_{{{\bf a}:=\{a_1,\ldots,a_s\}:{\sum a_i=r}}}J_{r,l,{\bf a}},$$ 
where 
$J_{\bf a}:=J_{r,s,{\bf a}}$ equals one if the $s$ integers in {\bf a} that sum to $r$ (with possible repetition) are all selected to be in $\ca$.  We then use Janson's inequality and the worst case scenario estimate for the $\Delta$ that arises to estimate
\[\e(I_jI_\ell)\le\exp\{-(\vert S_j\vert+\vert S_\ell\vert)p^k(1+O(1/np))+O(n^{2k-3}p^{2k-1})\}\]
and then conclude that
\begin{eqnarray*}
&&\sum_j\sum_{\ell\ne j} (\e(I_jI_\ell)-\e(I_j)\e(I_\ell))\\
&&{}\le \sum_j\sum_{\ell\ne j}\exp\{-(\vert S_j\vert+\vert S_\ell\vert)p^k(1+O(1/np))\}\lr e^{O(n^{2k-3}p^{2k-1})}-1\rr\\
&&\sim n^{2k-3}p^{2k-1}\l^2,
\end{eqnarray*}
so that 
\[\frac{\sum_j\sum_{\ell\ne j} (\e(I_jI_\ell)-\e(I_j)\e(I_\ell))}{\l}\le n^{2k-3}p^{2k-1}\l\to0\]
if $\l\ll1/n^{2k-3}p^{2k-1}$, i.e., if 
\[{\frac{2(k-2)!k!n^{k-1}p^{k-1}}{\a^{k-2}}\exp\{-\a^{k-1}n^{k-1}p^k/(k-1)!k!\}}\to 0,\]
i.e., if $p=\sqrt[k]{\frac{[\beta K]\log n}{n^{k-1}}}$, where $\beta>(k-1)/k$, as claimed.  \end{proof}

\noindent{\bf Open Problems}	
In both the modular and truncated cases, the representation function question has yet to be addressed.  We have preliminary results in this direction and plan on publishing them in a future paper.
\vspace{3mm}

\noindent{\bf Acknowledgment}  The research of AG, CML, and NGT was supported by NSF grant 1004624.  The collaboration between AG and VL was made possible due to the support of the Acheson J. Duncan Fund for the
Advancement of Research in Statistics.  VL was supported by U.S. Department of Education GAANN grant P200A090128.  We thank the anonymous referee for useful comments and for making us aware of key references (especially the work of Sandor \cite{sandor}).  The paper has improved significantly as a result.

\end{document}